\documentclass[a4paper, 11pt, final,underruledhead]{paper}

\usepackage[mathscr]{euscript}
\usepackage{amssymb, amsmath, amsthm, tabularx, booktabs, float}
\usepackage{enumitem, uwbgeom}

\newcolumntype{C}{>{\centering\arraybackslash}X}

\title{Geometry on the lines of spine spaces}
\author{K. Petelczyc and M. \.Zynel}

\def\LineOn(#1,#2){\overline{{#1},{#2}\rule{0em}{1,5ex}}}
\def\copla{\mathrel{\boldsymbol{\pi}}}
\def\copenc{\mathrel{\boldsymbol{\rho}}}
\def\cop{\mathrel{\boldsymbol{\delta}}}
\def\pcopla{\mathrel{\boldsymbol{\Pi}}}

\def\agen#1#2{{\mbox{{\boldmath$[\mkern-4.5mu|$}}{#1}\mbox{{\boldmath$|\mkern-4.5mu]$}}_{{#2}}}}

\def\penc{{\bf p}}

\def\TRG{\boldsymbol{\Delta}}

\def\fixspace{\ensuremath{\goth M}}

\def\eUps{{\Upsilon_{\emptyset}}}

\newcommand{\overbar}[1]{\mkern 2mu\overline{\mkern-2mu#1\mkern-2mu}\mkern 2mu}
\newcommand\cl[1]{\overbar{#1}}

\def\pencpeki{\mathcal P}
\def\peki{\mathcal{P}}

\def\W{{\cal H}}
\def\A{\mathfrak{A}}
\def\S{\mathcal{S}}
\def\T{\mathcal{T}}
\def\M{{\mathfrak M}}
\def\kliki{{\mathscr K}}

\def\PencSpace(#1,#2){{\bf P}_{#2}({#1})}
\def\linesout{{\Lines_{\W}}}
\def\pointsout{{S_{\W}}}

\def\semibundle{{\mathscr B}}

\DeclareMathOperator{\LinesOf}{L}

\DeclareMathOperator{\starof}{S}
\DeclareMathOperator{\topof}{T}

\def\bLinesOf(#1){\LinesOf\bigl(#1\bigr)}

\def\klamerka(#1, #2){\Lambda_{#2}(#1)}
\def\bklamerka(#1, #2){\Lambda_{#2}\bigl(#1\bigr)}

\def\aflines{{\cal A}}

\newenvironment{cmath}{%
  \par
  \smallskip
  \centering
  $
}{%
  $
  \par
  \smallskip
  \csname @endpetrue\endcsname
}

\def\ecase#1{\noindent$\bullet\ #1$\par}

\begin{document}

\maketitle

\begin{abstract}
  Spine spaces can be considered as fragments of a projective Grassmann space.
  We prove that the structure of lines together with binary coplanarity
  relation, as well as with binary relation of being in one pencil of lines, is
  a sufficient system of primitive notions for these geometries.  It is also
  shown that, over a spine space, the geometry of pencils of lines  can be
  reconstructed in terms of the two binary relations.
\end{abstract}

\begin{flushleft}\small
  Mathematics Subject Classification (2010): 51A15, 51A45.\\
  Keywords: Grassmann space, projective space, spine space, coplanarity
\end{flushleft}

%
\section*{Introduction}

It was proved in \cite{bundle} that the structure of lines together with binary 
coplanarity relation $\copla$ is a sufficient system of primitive notions for projective 
and polar Grassmann spaces. Following the ideas,  motivations, and objectives of that
paper we use similar methods to prove the same for spine spaces. This is also the reason
why our terminology and notation is heavily based on \cite{bundle}. In consequence many 
definitions and concepts may sound familiar for the reader. 

A relation close in its nature to the relation $\copla$ is the relation $\copenc$ of 
being in one pencil of lines. This closeness is revealed simply by the fact that
$\mathord{\copenc}\subseteq\mathord{\copla}$ and causes that some of the reasonings
are common for $\copla$ and $\copenc$. There are differences though, mainly at the
technical level at the beginning up to the point where bundles of lines are constructed.
From that point onwards the reasoning is unified and we get that 
$\copla$ as well as $\copenc$ is a sufficient primitive notion for the geometry
a spine space  $\fixspace$.
Generally, as in the case of Grassmann spaces, the key role play maximal cliques of 
$\copla$ and $\copenc$, as well as maximal strong subspaces containing them. 
The structure of strong subspaces in spine spaces is much more complex than in 
Grassmann spaces, but pretty well known if we take a look into \cite{autspine}, 
\cite{affspine}, \cite{extparal}, and \cite{lincomp}.
The major difference is that we have to deal with three types of lines and
four types of strong subspaces as stars and tops can be projective or semiaffine 
(in particular affine) spaces.

The two relations $\copla$ and $\copenc$ are mainly used to reconstruct the
pointset of a spine space $\fixspace$, but they also appear to be sufficient
primitive notions for the geometry of pencils of lines over $\fixspace$. Both of
them are trivial on planes so, it is natural to require that all maximal strong
subspaces of $\fixspace$ are at least 3-dimensional. This is all we need to
express the geometry of pencils in terms of $\copla$ and $\copenc$.

Our approach to the reconstruction of $\fixspace$, based on bundles of its
lines,  requires a bit different assumption though, that stars or tops are at
least 4-dimensional. In some cases excluded by this assumption, which we point
out in the last section, the geometry $\fixspace$ can be recovered.  The most
interesting case is the neighbourhood of a point in the underlying Grassmann
space. It is impossible to recover this spine space from the structure of lines
equipped with neither $\copla$ nor $\copenc$, nonetheless the geometry of
its lines and pencils can be reconstructed. This shows that the geometry of lines
and pencils over $\fixspace$ is intrinsically weaker than the geometry of
$\fixspace$.

In the vein of Chow's Theorem one would want to continue the procedure of
creating  pencils over pencils. Sadly, the geometry of pencils of planes,
obtained in the first step, turns out to be disconnected which makes the whole
idea pointless.

The Appendix~\ref{sec:erratum} at the end of this paper fixes a gap
in the proof of Lemma~1.3 in \cite{bundle}.


\section{Generalities}

A point-line structure $\A=\struct{S, \Lines}$, where the elements of $S$ are
called \emph{points}, the elements of $\Lines$ are called \emph{lines}, and
where $\Lines\subset2^S$, is said to be a \emph{partial linear space}, or 
a \emph{point-line space}, if two
distinct lines share at most one point and every line is of size (cardinality) 
at least 2 (cf.\ \cite{cohen}).

A \emph{subspace} of $\A$ is any set $X\subseteq S$ with the property that
every line which shares with $X$ two or more points is entirely contained in $X$.
We say that a subspace $X$ of $\A$ is \emph{strong} if any two points in $X$ are
collinear.
A \emph{plane} in $\A$ is a strong subspace $E$ of $\A$ with the property that 
the restriction of $\A$ to $E$ is a projective plane. 
If $S$ is strong, then $\A$ is said to be a \emph{linear space}. 

Let us fix nonempty subset $\W\subset S$ and consider two sets
  $$\pointsout := S\setminus\W\qquad\text{and}\qquad
     \linesout := \bigl\{ k\cap\pointsout\colon k\in\Lines \text{ and } 
       \abs{k\cap \pointsout}\ge 2 \bigr\}.$$
The structure
  $$\M := \struct{\pointsout, \linesout}$$
is a \emph{fragment} of $\A$ and itself it is a partial linear space.
The incidence relation in $\M$ is again $\in$, inherited from $\A$, but
limited to the new pointset and lineset.
Following a standard convention we call the points and lines of $\M$ \emph{proper}, 
and those points and lines
of $\A$ that are not in $\pointsout$, $\linesout$ respectively are said to be 
\emph{improper}. The set $\W$ will be called the \emph{horizon of\/ $\M$}.
To every line $L\in\linesout$ we can assign uniquely the line $\cl{L}\in\Lines$,
the \emph{closure} of $L$, such that $L\subseteq\cl{L}$. 
For a subspace $X\subseteq\pointsout$ the closure of $X$ is the minimal 
subspace   $\cl{X}$ of $\A$ containing $X$. 
A line $L\in\linesout$ is said to be a \emph{projective line} if $L=\cl{L}$, 
and it is said to be an \emph{affine line} if $\abs{\cl{L}\setminus L} = 1$. 
In case $\linesout$ contains projective or affine lines only, 
then $\M$ is a \emph{semiaffine} geometry
(for details on terminology and axiom systems see \cite{kradzisz} and \cite{kradzisz2}).
In this approach an affine space is a particular case of a
semiaffine space. 
For affine lines $L_1,L_2\in\linesout$ we can define parallelism in a natural way:
\begin{cmath}
  L_1, L_2\ \text{are \emph{parallel}\enspace iff}\enspace \cl{L_1}\cap\cl{L_2}\in\W. 
\end{cmath}
We say that $E$ is a plane in $\M$ if $\cl{E}$ is a plane in $\A$.
Observe that there are two types of planes in $\M$: projective and semiaffine.
A semiaffine plane $E$ arises from $\cl{E}$ by removing a point
or a line. 
In result we get a punctured plane or an affine plane respectively.
For lines $L_1, L_2\in\linesout$ we say that they are \emph{coplanar} and write
\begin{equation}\label{eq:def:copla}
  L_1\copla L_2\quad\text{iff}\quad \text{there is a plane } E \text{ such that } L_1, L_2\subset E
\end{equation}
Let $E$ be a plane in $\M$ and $U\in\cl{E}$. A set 
\begin{equation}
  \penc(U,E):=\bset{L\in\linesout\colon U\in\cl{L}\subseteq\cl{E}}
\end{equation}
will be called a \emph{pencil of lines} if $U$ is a proper
point, or a \emph{parallel pencil} otherwise. The point $U$ is said to be the
\emph{vertex} and the plane $E$ is said to be the \emph{base plane} of that pencil.
We write
\begin{equation}\label{eq:def:copenc}
  L_1\copenc L_2\quad\text{iff}\quad \text{there is a pencil } p \text{ such that } L_1,L_2\in p.
\end{equation}
If $L_1\copenc L_2$, then clearly $L_1\copla L_2$. This means that every $\copenc$-clique is a $\copla$-clique.

For a subspace $X$ of $\M$ we write
\begin{equation}\label{eq:def:flat}
  \LinesOf(X) = \set{L\in\linesout\colon L\subset X}.
\end{equation}

If $E$ is a plane in $\M$, then the set $\LinesOf(E)$ will be called a \emph{flat}.
The set of all projective lines on $E$ augmented with a maximal set of affine lines 
on $E$ such that no two are parallel will be called a \emph{semiflat}.
As the plane $E$ can be projective, punctured or affine, we have
projective, punctured or affine semiflat (flat) respectively.
Semiflats that are not projective will be called semiaffine semiflats.

Note that projective semiflat is a projective flat. On an affine plane $E$,
where parallelism partitions the lineset into directions, a semiflat 
is a selector of $\linesout/_{\mathord{\parallel}}$
(cf.\ \cite{mafodja:dim1}).

For a subspace $X$ of $\M$, if $U\in\cl{X}$ we write
\begin{equation}\label{eq:def:semibundle}
  \LinesOf_U(X) = \set{L\in\linesout\colon U\in\cl{L} \text{ and }  L\subseteq X}.
\end{equation}

If $X$ is a strong subspace of $\M$, then $\LinesOf_U(X)$ is said to be a \emph{semibundle}.
As the vertex $U$ of a semibundle $\LinesOf_U(X)$ can be proper or improper, we call
the semibundle \emph{proper} or \emph{improper} accordingly. We omit the adjective
when we mean a semibundle in general.

\subsection{Grassmann spaces}

Let $V$ be a vector space of dimension $n$ with $3\le n<\infty$. The set of all
subspaces of $V$ will be written as $\Sub(V)$ and the set of all $k$-dimensional
subspaces (or $k$-subspaces in short) as $\Sub_k(V)$. By
 a \emph{$k$-pencil} we call the set of the form
  $$\penc(H, B) := \{ U\in\Sub_k(V)\colon H\subset U\subset B\},$$
where $H\in\Sub_{k-1}(V)$, $B\in\Sub_{k+1}(V)$, and $H\subset B$.
The family of all such $k$-pencils will be denoted by $\peki_k(V)$.
A \emph{Grassmann space} (also known as a \emph{space of pencils} or a \emph{projective Grassmannian}) is a point-line space
  $$\PencSpace(V, k) = \struct{\Sub_k(V), \peki_k(V)},$$
 with $k$-subspaces of $V$ as points and $k$-pencils as lines
(see \cite{polargras}, \cite{slitgras} for a more general definition, see also \cite{mark}).
For $0 < k < n$ it is a partial linear space. For $k=1$ and
$k=n-1$ it is a projective space. So we assume that
  $$1< k < n-1.$$

It is known that there are two classes of maximal strong subspaces in $\PencSpace(V, k)$:
\emph{stars} of the form
  $$\starof(H) = [H)_k = \{U\in\Sub_k(V)\colon H\subset U\},$$
where $H\in\Sub_{k-1}(V)$, and \emph{tops} of the form
  $$\topof(B) = (B]_k = \{U\in\Sub_k(V)\colon U\subset B\},$$
where $B\in\Sub_{k+1}(V)$.
Although non-maximal stars $[H, Y]_k$ and non-maximal tops $[Z, B]_k$, for
some $Y, Z\in\Sub(V)$, make sense but in this paper when we say `a star' or `a top'
we mean a maximal strong subspace.
It is trivial that every line, a $k$-pencil $p=\penc(H, B)$, of $\M$ can be 
uniquely extended to the star $\starof(p):=\starof(H)$ and to the top 
$\topof(p):=\topof(B)$.

\subsection{Spine spaces}\label{sec:spinespaces}

A \emph{spine space} is a fragment of a Grassmann space
chosen so that it consists of subspaces of $V$ which meet a fixed subspace in
a specified way. The concept of spine spaces was introduced in \cite{spinesp} and
developed in \cite{affspine}, \cite{autspine}, \cite{lincomp}, \cite{extparal}.

Let $W$ be a fixed subspace of $V$ and let $m$ be an integer with 
\begin{equation}\label{eq:spineparam}
  k-\codim(W)\le m\le k, \dim(W).
\end{equation}  
From the points of the Grassmann space $\SpPencils(k, V)$ we take those which as subspaces of $V$
meet $W$ in dimension $m$, that is:
  $$\SpineSpPoints(k,m,W) := \set{U\in\Sub_k(V)\colon \dim(U\cap W) = m}.$$
As new lines we take those lines of $\SpPencils(k, V)$ which have at least two new points:
  $$\SpineSpLines(k,m,W) := \set{L\cap\SpineSpPoints(k,m,W)\colon L\in\SpPencilsLines(k, V) \text{ and }
      \abs{L\cap\SpineSpPoints(k,m,W)} \ge 2 }.$$
\noindent
The point-line structure:
  $$\M = \SpineSp(k,m,V,W) := \bstruct{\SpineSpPoints(k,m,W),\/ \SpineSpLines(k,m,W)}$$
will be called a \emph{spine space}. This is a Gamma space. Specifically, depending on
$k, m$ and $\dim(W)$ it can be: a projective space, a slit space (cf.\ \cite{KM67}, \cite{KP70}),
an affine space or the  \emph{space of linear complements} (cf.\ \cite{blunck-havlicek}, \cite{lincomp}).
As $\M$ is a fragment of the Grassmann space $\SpPencils(k, V)$ 
we can distinguish a set $\W$ of improper points in $\M$, i.e.\ a horizon. 
Consequently, a line of $\M$ is either affine or projective.

The class of affine lines is denoted by $\aflines$. Projective lines fall into two disjoint 
classes $\Lines^\alpha$ and $\Lines^\omega$. 
For brevity $\Lines := \aflines\cup\Lines^\alpha\cup\Lines^\omega$.
Details can be found in Table~\ref{tab:lines}. 

\begin{table}[H]
  \small%
  \begin{center}
    \begin{tabular}{@{}lcc@{}}\toprule%
      \multicolumn{1}{c}{class} & representative line $g=\Pencil(H, B)\cap\SpineSpPoints(k,m,W)$  & $g^\infty$ \\
      \midrule
      $\AfLines(k,m,W)$ & $H\in\SpineSpPoints(k-1,m,W),\ B\in\SpineSpPoints(k+1,m+1,W)$       & $H+(B\cap W)$ \\
      $\aLines(k,m,W)$  & $H\in\SpineSpPoints(k-1,m,W),\ B\in\SpineSpPoints(k+1,m,W)$         & -- \\
      $\oLines(k,m,W)$  & $H\in\SpineSpPoints(k-1,m-1,W),\ B\in\SpineSpPoints(k+1,m+1,W)$     & -- \\
      \bottomrule
    \end{tabular}
    \caption{The classification of lines in a spine space \SpineSp(k,m,V,W).}
    \label{tab:lines}    
  \end{center}  
\end{table}

The geometry of a spine space is complex in that there is an overwhelming variety
of types  of subspaces. As usual most important are strong subspaces. 
A star from $\SpPencils(k, V)$ restricted to $\M$ either contains affine lines or not. 
In the first case it is called an $\alpha$-star which is a semiaffine space, in the other 
case it is called an $\omega$-star which is a projective space. A top from $\SpPencils(k, V)$ 
restricted to $\M$ also contains affine lines or not and is called an $\omega$-top 
or an $\alpha$-top respectively. 
On the other hand, each strong subspace $X$ of a spine space is a slit space,
that is a projective space $\mathbf P$ with a subspace $\cal D$ removed. The
form of a maximal strong subspace from each class and the dimension of the
corresponding spaces $\mathbf P$ and $\cal D$ are presented in 
Table~\ref{tab:subs}. In the extremes $\cal D$ can be void, then $X$ is
basically  a projective space, or a hyperplane, then $X$ is an affine space.
One should be aware that for specific values of $\dim(V)$, $\dim(W)$, $k$, and $m$
some classes of maximal strong subspaces are void.

\begin{table}[H]
  \begin{center}
    \noindent\small%
    \begin{tabular}{@{}lccc@{}}\toprule
      \multicolumn{1}{c}{class} & representative subspace
        & $\dim({\mathbf P})$  & $\dim({\cal D})$ \\
      \midrule
      $\omega$-stars   & $\seg{H}{H+W}_k\colon H\in\SpineSpPoints(k-1,m-1,W)$
        & $\dim(W)-m$         & -1 \\
      $\alpha$-stars  & $\seg{H}{V}_k\cap\SpineSpPoints(k,m,W)\colon H\in\SpineSpPoints(k-1,m,W)$
        & $\dim(V)-k$         &  $\dim(W)-m-1$ \\
      $\alpha$-tops   & $\seg{B\cap W}{B}_k\colon B\in\SpineSpPoints(k+1,m,W)$
        & $k-m$              & -1 \\
      $\omega$-tops  & $\seg{\Theta}{B}_k\cap\SpineSpPoints(k,m,W)\colon B\in\SpineSpPoints(k+1,m+1,W)$
        & $k$                & $k-m-1$ \\
      \bottomrule
    \end{tabular}
    \caption{The classification of stars and tops in a spine space \SpineSp(k,m,V,W).}
    \label{tab:subs}    
  \end{center}  
\end{table}

A line of $\M$ can be in at
most two maximal strong subspaces of different type: a star and a top.

\begin{fact}\label{fact:spineintersections}
  A projective star and a projective top are either disjoint or they share a point. 
  In remaining cases, a star and a top are either disjoint or they share a line.
  Two stars (or two tops) are either disjoint or they share a point.
\end{fact}

\begin{lem}\label{lem:spineplanes}
  Three pairwise coplanar and concurrent, or parallel, lines not all on a plane span a star or a top.
\end{lem}

\begin{proof}
  There are three planes and it suffices to note by
  \ref{fact:spineintersections} that no two of them can be of distinct type,
  i.e.\ they all are of type star or top. Consequently, all these lines lie
  in one maximal strong subspace.
\end{proof}


\section{Maximal cliques}

Let $\M$ be a spine space. 
The goal now is to show that the set of lines equipped with
either coplanarity relation $\copla$ or relation $\copenc$ of being in one pencil,
is a sufficient system of primitive notions for $\M$.
The key tool to achieve that are maximal cliques.

\begin{lem}\label{lem:maxcliques}
  \begin{sentences}
  \item
    Flats and semibundles are $\copla$-cliques.
  \item
    Semiflats and proper semibundles are $\copenc$-cliques.
  \end{sentences}
\end{lem}

\begin{proof}
  (i): 
  It is clear that flats are $\copla$-cliques. Let $X$ be a strong subspace of
  $\M$ and $U\in\cl{X}$. Note that $X$ is, up to an isomorphism, a slit
  space. Therefore, any two lines in the semibundle $\LinesOf_U(X)$  are coplanar.

  (ii): 
  Let $K$ be a semiflat. If $K$ is a projective semiflat, then every two lines
  of $K$ are concurrent. In case $K$ is a punctured semiflat, there is a single
  affine line in $K$ which intersects all the other lines in $K$.
  If $K$ is an affine semiflat, then all the lines in $K$ are affine and
  any two of them are concurrent. This justifies that $K$ is a $\copenc$-clique. 
  The fact that that proper semibundles are $\copenc$-cliques is evident.
\end{proof}

\begin{prop}\label{lem:max-copla-cliques}
  Every maximal $\copla$-clique is either a flat or a semibundle.
\end{prop}

\begin{proof}
  Let $K$ be a maximal $\copla$-clique which is not a flat.
  So, there are three pairwise distinct lines in $K$ not all on a plane. They all
  meet in a point $U$, possibly improper. Moreover, for any $L\in K$ we have $U\in\cl{L}$.

  Let $U_1, U_2$ be distinct points in $\bigcup K$. There are lines $M_1, M_2\in K$
  such that $U_1\in M_1$, $U_2\in M_2$. As $M_1\copla M_2$ the points $U_1, U_2$
  are collinear. This means that $\bigcup K$ is a collinearity clique.
  
  Now, let $\cl{\M}$ be the Grassmann space embracing $\M$,
  and let $\cl{X}$ be a maximal collinearity clique in $\cl{\M}$ containing 
  $\bigcup K$. Set $X := \cl{X}\cap\SpineSpPoints(k,m,W)$. It is clear that $\bigcup K\subseteq X$.
  As  $\cl{\M}$ is a Gamma space $\cl{X}$ is a maximal strong subspace of 
  $\cl{\M}$. Therefore, $X$ is a maximal strong subspace of $\M$. 
  Note that $U\in\cl{X}$. 
  Take a point $W\in X$ distinct from $U$. There is a, possibly affine, line $M=\LineOn(U, W)$ 
  contained in $\cl{X}$. Since $\cl{X}$ is, up to an isomorphism, a projective space, 
  all the lines through $U$ in $\cl{X}$ are pairwise coplanar, so $M\in K$. Hence
  $W\in\bigcup K$ and consequently $\bigcup K = X$. We have actually shown that
  $K=\LinesOf_U(X)$ which means that $K$ is a semibundle.
\end{proof}

\begin{prop}\label{lem:max-copenc-cliques}
  Every maximal $\copenc$-clique is either a semiflat or a proper semibundle.
\end{prop}

\begin{proof}
  Let $K$ be a maximal $\copenc$-clique which is not a semiflat.
  Every $\copenc$-clique is a $\copla$-clique, so $K$ is a $\copla$-clique though not
  necessarily maximal. Let $K'$ be a maximal $\copla$-clique containing $K$.
  By \ref{lem:max-copla-cliques} $K'=\LinesOf_U(X)$ for some maximal strong subspace 
  $X$ and a proper point $U$. As $X$ is a projective space and $K\subseteq K'$ we 
  get $K=K'$ and the claim follows. 
\end{proof}

\begin{lem}\label{rem:semiflat}
  A maximal $\copenc$-clique $K$ satisfies the following condition:
  \begin{multline}\label{eq:podmianka}
    \text{there are lines } L_1\in K, L_2\in\Lines\setminus K
    \text{ such that } \\
    (K\setminus\set{L_1})\cup\set{L_2} \text{ is a maximal $\copenc$-clique}
  \end{multline}
  iff $K$ is a semiaffine semiflat.
\end{lem}

\begin{proof}
\ltor 
Let $K$ be a maximal $\copenc$-clique that satisfies \eqref{eq:podmianka} and
is not a semiaffine semiflat. By \ref{lem:max-copenc-cliques}, $K$ is either a 
projective flat or a proper semibundle. In both cases the unique line that complements
the set $K\setminus\set{L_1}$ to a maximal $\copenc$-clique is $L_1$. This
contradicts \eqref{eq:podmianka}.

\rtol
Let $E$ be a semiaffine plane and $K$ be a semiflat on $E$.
Take two affine lines: $L_1\in K$ and $L_2\subseteq E$ such that 
$L_2\parallel L_1$, $L_2\neq L_1$. Note that $L_2\notin K$. Then the set
of affine lines in $K$ with $L_1$ replaced by $L_2$ is a
maximal set of affine lines in $E$ such that no two are parallel. Consequently,
$(K\setminus\set{L_1})\cup\set{L_2}$ is a maximal $\copenc$-clique.
\end{proof}

The criterion \eqref{eq:podmianka} from \ref{rem:semiflat} could be used to distinguish
semiaffine semiflats in the family of all $\copenc$-cliques.

Let $\cop\;\in\{\copla, \copenc\}$. As indispensable as the property of the family 
of maximal $\cop$-cliques provided by \ref{lem:max-copla-cliques} and 
\ref{lem:max-copenc-cliques} is the characterization of this family in terms of lines
and the relation $\cop$, that is an elementary definition of maximal $\cop$-cliques
within~$\struct{\Lines,\cop}$.
For lines $L_1,L_2,\dots,L_n\in\Lines$ we define
\begin{multline}\label{eq:trgcopla}
  \TRG_n^{\cop}(L_1,L_2,\dots, L_n) 
    \quad\text{iff}\quad 
    \neq(L_1,L_2,\dots, L_n)\enspace\text{and}\enspace
    L_i\cop L_j \text{ for all } i,j=1,\dots, n,\\
    \text{and}\enspace
    \text{for all}\ M_1, M_2\in\Lines\ \text{if}\ M_1, M_2\cop L_1,L_2,\dots,L_n\ \text{then}\ M_1\cop M_2.
\end{multline}
Note that the relation $\TRG_2^{\cop}$ is empty. 
In spine spaces $3$ lines satisfying $\TRG_3^{\cop}$ 
determine a $\cop$-clique.
From this point of view, $\TRG_n^{\cop}$ could be defined as a ternary relation in \eqref{eq:trgcopla}.
The reason why we introduce so general formula will be explained 
later in Erratum~\ref{sec:erratum}.

As the relation $\cop$, and in consequence $\TRG_3^{\cop}$, makes a little sense on a plane
we  assume that every plane in $\fixspace$ is contained is a star or top of
dimension at least 3. In view of Table~\ref{tab:subs} it reads as follows:

\begin{equation}\label{eq:3-spaces-num}
  3\le n-k \quad\text{and}\quad 3\le k-m.
\end{equation}

\begin{lem}\label{lem:TRG}
  Let\/ $L_1, L_2, L_3\in\Lines$.
  \begin{sentences}
  \item
     $\TRG_3^{\copla}(L_1,L_2,L_3)$ iff\/ $\cl{L_1}, \cl{L_2}, \cl{L_3}$ form 
     a tripod or a triangle. 
  \item
     $\TRG_3^{\copenc}(L_1,L_2,L_3)$ iff\/ $L_1, L_2, L_3$ form a $\copenc$-clique, 
     they are not in a pencil of lines, they are not on an affine plane, and 
     in case they are on a punctured plane one of\/ $L_1, L_2, L_3$ is an affine line.
  \end{sentences}
\end{lem}

\begin{proof}
  (i): Observe that $L_1,L_2,L_3$ are pairwise coplanar iff $\cl{L_1}, \cl{L_2}, \cl{L_3}$ 
  form a tripod, a triangle or a pencil of lines. If the later, a line $M_1$
  through the vertex of the pencil but not on its base plane is not coplanar with
  a line $M_2$ that lies on the base plane but misses the vertex.
  
  (ii): A pencil of lines is singled out taking the same lines $M_1, M_2$ as in (i). 
  On an affine plane, as well as on a punctured plane if none of $L_1, L_2, L_3$ goes 
  through the point at infinity, the lines $M_1, M_2$ from \eqref{eq:trgcopla} can be parallel.
\end{proof}

Provided that $\TRG_3^{\cop}(L_1,L_2, L_3)$, the maximal $\cop$-clique spanned by 
the lines $L_1, L_2, L_3$ is the set
\begin{equation}
  \agen{L_1,L_2,L_3}{\cop} := \bset{L\in \Lines\colon L\cop L_1,L_2, L_3}.
\end{equation}
Note that all maximal $\cop$-cliques except affine semiflats can be spanned in this way.
Now, let us define
 \begin{equation}\label{eq:def:kliki}
  \kliki_{\cop} = \bset{ \agen{L_1,L_2,L_3}{\cop}\colon 
    L_1,L_2, L_3\in\Lines\ \text{ and }\ \TRG_3^{\cop}(L_1,L_2, L_3) }.
\end{equation}
Lemmas \ref{lem:max-copla-cliques} and \ref{lem:max-copenc-cliques} explain which maximal
$\cop$-cliques fall into $\kliki_{\cop}$.
\begin{lem}\label{lem:kliki}
  \begin{sentences}
  \item
     $\kliki_{\copla}$ is the family of all maximal $\copla$-cliques, i.e.\ flats and semibundles.
  \item
     $\kliki_{\copenc}$ is the family of all projective flats, punctured semiflats 
     and proper semibundles.
  \end{sentences}
\end{lem}

In consequence we get the following.

\begin{prop}\label{lem:coplacliques}
  The family of maximal $\cop$-cliques is definable in $\struct{\Lines,\cop}$.
\end{prop}


\section{Proper semibundles}

Let us recall that our goal is to reconstruct the point universe of a spine space
given a line universe equipped with the relation $\cop$. The idea is to use
vertices of semibundles to do that. This means that only proper semibundles
are of our concern. 
The problem is we need to distinguish them in the family of all $\cop$-cliques,
which we are going to do now using pencils of lines as an essential tool.

\subsection{Pencils of lines}

The geometry induced by pencils of lines is interesting in itself and, we believe,
it is worth to give it a little more attention here.
For three lines $L_1,L_2,L_3\in\Lines$ we make the following two definitions:
\begin{multline}\label{eq:ternaryconc:copla}
  \penc_{\copla}(L_1,L_2,L_3)
    \quad\text{iff}\quad
    L_i\copla L_j \text{ for all } i,j=1,2,3\enspace\text{and}\enspace
    \neg\TRG_3^{\copla}(L_1,L_2, L_3),
\end{multline}
\vspace*{-6ex}
\begin{multline}\label{eq:ternaryconc:copenc}
  \penc_{\copenc}(L_1,L_2,L_3)
    \quad\text{iff}\quad 
      \text{there are}\ M_1, M_2, M_3\in\Lines\ \text{such that}\\ 
      \TRG_3^{\copenc}(M_1,M_2, M_3)\ \text{and}\ \agen{M_1,M_2,M_3}{\copenc}\ \text{does not satisfy}\ \eqref{eq:podmianka}\\
      \text{and}\enspace L_1,L_2,L_3\in\agen{M_1,M_2,M_3}{\copenc}\enspace\text{and}\enspace\neg\TRG_3^{\copenc}(L_1,L_2, L_3). 
\end{multline}

\begin{lem}\label{lem:penc}
  Let\/ $L_1, L_2, L_3\in\Lines$.
  \begin{sentences}
  \item
     $\penc_{\copla}(L_1,L_2,L_3)$ iff\/ $L_1, L_2, L_3$ form a pencil of lines or a parallel pencil. 
  \item
     $\penc_{\copenc}(L_1,L_2,L_3)$ iff\/ $L_1, L_2, L_3$ form a pencil of lines.     
  \end{sentences}
\end{lem}

\begin{proof}
  (i): Immediate by \ref{lem:TRG}(i).
  
  (ii) 
  \ltor For $M_1, M_2, M_3$ in \eqref{eq:ternaryconc:copenc} let $K := \agen{M_1,M_2,M_3}{\copenc}$. 
  Note that $K\in\kliki_{\copenc}$, but in view of \ref{rem:semiflat}, it is not a punctured 
  semiflat. So, by \ref{lem:kliki} $K$ is a projective flat or a proper semibundle that is a maximal $\copenc$-clique.
  If $L_1,L_2,L_3\in K$ and $\neg\TRG_3^{\copenc}(L_1,L_2, L_3)$, then by \ref{lem:TRG}(ii) 
  the lines $L_1,L_2,L_3$ form a pencil of lines.
  
  \rtol
  Assume that $L_1,L_2,L_3$ form a pencil of lines, $U$ is its vertex and $E$ is its base plane.
  If $E$ is projective, then as $M_1, M_2, M_3$ any triangle on $E$ can be taken. 
  If $E$ is semiaffine, then by \eqref{eq:3-spaces-num} the plane $E$ is extendable to some 
  star or top $X$ of dimension at least 3 and as $M_1, M_2, M_3$ a tripod in $\LinesOf_U(X)$ should be taken.
\end{proof}

Pencils can be defined in a standard way
using ternary concurrency provided by \eqref{eq:ternaryconc:copla} or \eqref{eq:ternaryconc:copenc}
so, we claim that:
\begin{lem}\label{lem:allpencils}
  \begin{sentences}
  \item
    The family $\pencpeki_{\copla}$ of all pencils of lines and parallel pencils
    is definable in $\struct{\Lines, \copla}$.
  \item
    The family $\pencpeki_{\copenc}$ of all pencils of lines is definable in $\struct{\Lines, \copenc}$.
  \end{sentences}
\end{lem}

\subsection{Parallel pencils}

As proper semibundles contain no parallel pencils we need to get rid of them from 
$\pencpeki_{\copla}$. 
It is a bit tricky however to express that the vertex of a pencil is improper in terms of 
$\struct{\Lines, \pencpeki_{\copla}}$. 
Affine planes and punctured planes, base planes of pencils in $\pencpeki_{\copla}$,
have to be treated separately.

Two pencils $p_1, p_2\in\pencpeki_{\copla}$ are \emph{coplanar} iff every two lines
$l_1\in p_1$, $l_2\in p_2$ are coplanar. That is: let $p_1, p_2\in\pencpeki_{\copla}$
\begin{equation}
  p_1\pcopla p_2 \iff \text{for all } l_1\in p_1, l_2\in p_2\ \text{ we have } l_1\copla l_2.
\end{equation}

On an affine plane parallel pencils are those pencils that contain the line at infinity.
So, formally, a pencil $p_1\in\pencpeki_{\copla}$ is a parallel pencil if there is another 
pencil $p_2\in\pencpeki_{\copla}$ such that $p_1\pcopla p_2$ and $p_1\cap p_2=\emptyset$.
Hence, a plane in $\struct{\Lines, \pencpeki_{\copla}}$ is affine iff it contains two disjoint pencils.
A pencil $p\in\pencpeki_{\copla}$ lies on an affine plane iff there are two distinct pencils 
$p_1, p_2\in\pencpeki_{\copla}$
such that $p_1\pcopla p$, $p_2\pcopla p$ and $p_1\cap p_2=\emptyset$.
We say that a line $l$ lies on an affine plane iff there is a pencil $p\in\pencpeki_{\copla}$ 
such that $l\in p$ and $p$ lies on an affine plane.

If the base plane of a pencil $p\in\pencpeki_{\copla}$ is not affine but every line $l\in p$ 
lies on some affine plane, then the vertex of $p$ is an improper point. If that is the case
the pencil $p$ is a parallel pencil and its base plane is a punctured plane.

Finally, we have proved that:

\begin{lem}\label{lem:parpencils}
  The family $\pencpeki_\parallel$ of all parallel pencils is definable in $\struct{\Lines, \copla}$.
\end{lem}

Let $\pencpeki$ be the family of all pencils of lines in $\fixspace$.
By \ref{lem:allpencils}(i) and \ref{lem:parpencils} the family 
$\pencpeki=\pencpeki_{\copla}\setminus\pencpeki_\parallel$ is definable for $\copla$.
By \ref{lem:allpencils}(ii) 
the family $\pencpeki=\pencpeki_{\copenc}$ is definable for $\copenc$.
Note that two pencils from $\pencpeki$ are either disjoint or share a line.
This means that $\struct{\Lines, \pencpeki}$ is a partial linear space.
Interestingly enough, we are able to reconstruct this geometry, induced by pencils 
of lines on $\fixspace$, using nothing but its lines and the relation $\cop$.

\begin{prop}\label{prop:pencpenc}
  If\/ $\fixspace$ satisfies \eqref{eq:3-spaces-num}, then   
  $\struct{\Lines, \pencpeki}$ is definable in $\struct{\Lines,\cop}$.
\end{prop}

\subsection{Geometry induced by pencils of lines}

A $\copla$-clique is proper if it contains no parallel pencil while all $\copenc$-cliques
are proper.
The family of all proper maximal $\cop$-clique consists of projective flats, punctured
semiflats and proper semibundles. 
Every proper maximal $\cop$-clique together with pencils of lines it contains carries some geometry. 
A projective flat determines a projective plane, a punctured semiflat determines
a projective plane with all but one points on some line removed, and a proper semibundle determines
a projective space.
The geometrical dimension of a
proper flat is always 2 whereas a proper semibundle $\LinesOf_U(X)$ has dimension one less than
the dimension of $X$. This lets us distinguish proper flats from proper semibundles if we
assume that 
\begin{equation}\label{eq:4-spaces}
  \text{stars or tops in $\fixspace$ are at least 4-dimensional projective or semiaffine spaces}.
\end{equation}
Let us write $\pencpeki_0$ for the family of all pencils of lines definable
in $\struct{\Lines,\cop}$ and set
\begin{equation}\label{eq:kliki-penc}
  \kliki^0_{\cop} := \bset{ K\in\kliki_{\cop}\colon \text{there is } q\in\pencpeki_0  \text{ such that } q\subset K }.
\end{equation}
In case $\fixspace$ satisfies \eqref{eq:3-spaces-num} we have $\pencpeki_0=\pencpeki$
and $\kliki^0_{\cop} = \kliki_{\cop}$. Otherwise, $\pencpeki_0$ contains only those
pencils of lines and $\kliki^0_{\cop}$ only those maximal $\cop$-cliques which,
accordingly to \eqref{eq:4-spaces}, lie in 4-dimensional stars or tops.
Any way, for a $\cop$-clique $K\in\kliki^0_{\cop}$
we can define its geometrical dimension $\dim(K)$. This lets us make the following definition
\begin{equation}\label{eq:semibundle}
  \semibundle :=\bset{ K\in\kliki^0_{\cop}\colon \dim(K) \ge 3 }.
\end{equation}
Under assumptions \eqref{eq:4-spaces} it is the family of all proper semibundles
regardless which of the two relations $\copla$ or $\copenc$ we take. 
In the following lemma we state that more precisely.

\begin{lem}\label{lem:semibundle-flaty}
  The family $\semibundle$ defined in $\struct{\Lines,\cop}$ coincides with
  the family of all proper top semibundles, the family of all proper star semibundles or the
  union of these two families depending on whether tops, stars or all of them 
  as projective or semiaffine spaces are at least 4-dimensional.
\end{lem}


\section{Bundles}

On the family $\semibundle$ of proper semibundles we define

\begin{multline}\label{eq:upsilon}
  \Upsilon(K_1, K_2) \iff (\exists L_1, L_2\in K_1)(\exists M_1, M_2\in K_2)\\
     \bigl[L_1\neq L_2 \Land L_1\cop M_1 \Land L_2\cop M_2\bigr].
\end{multline}

\begin{lem}\label{lem:zlepsemibundle1}
Let $K_i := \LinesOf_{U_i}(X_i)\in \semibundle$, $i=1,2$.
If\/
$\Upsilon(K_1, K_2)$ and $K_1\cap K_2= \emptyset$,
then $X_1, X_2$ are both stars or tops  and\/ $U_1 = U_2$.
\end{lem}

\begin{proof}
By \eqref{eq:upsilon} there are lines $L_1\in K_1$ and $M_1\in K_2$, which are 
coplanar. 
Note that $X_1\neq X_2$ and $L_1\neq M_1$ as we assume that $K_1\cap K_2= \emptyset$.
Let $E_1$ be a plane containing $L_1,M_1$. Then $L_1\subseteq X_1,E_1$ 
and $M_1\subseteq X_2,E_1$. In view of \ref{fact:spineintersections} it means 
that 
$X_1, X_2$ are both of the same type and $E_1$ is of different type.

There is another pair of coplanar lines $L_2, M_2$ such that $L_2\in K_1$ 
and $M_2\in K_2$, since $\Upsilon(K_1, K_2)$. Let $E_2$ be the plane spanned 
by $L_2, M_2$. 
Note that $E_1, E_2$ are planes of the same type. As $U_1,U_2\in E_1,E_2$, 
by \ref{fact:spineintersections} we get that $E_1=E_2$ or $U_1=U_2$. If $E_1=E_2$, 
then $L_1,L_2,M_1,M_2\subseteq X_1, X_2$ which
yields a contradiction as $X_1,X_2$ are of the same type.
\end{proof}

The inverse of \ref{lem:zlepsemibundle1} is not true in general, 
which is manifested in the following example.
Let
$X_1$ be a semiaffine, but not affine, star. Then $X_1$ is an $\alpha$-star
(cf.\ \cite{spinesp}, \cite{affspine}, \cite{autspine}).  
Take a projective line $L\in K_1$.  The line
$L$ is an $\alpha$-line and the unique top-extension of $L$ is an $\alpha$-top
$Y$, that is a projective space. In case $X_2$ is an $\omega$-star, that is also
a projective space, there is no line in $Y\cap X_2$, provided by
\ref{fact:spineintersections}. Therefore we cannot find a line in $X_2$, which
is coplanar with $L$.
However this problem ceases to exist when we have at least two affine
lines in $K_1$. So
$X_1$ cannot be a punctured projective space as in such
there is only one affine line through a given proper point.

In view of Table~\ref{tab:subs} punctured projective spaces arise in a spine space 
as stars when $\dim(W)-m-1=0$ or as tops when $k-m-1=0$. Note that either,
all or none of the stars, and respectively, all or none of the tops,
are punctured projective spaces. For this reason we assume that
no star or no top is a punctured projective space, more precisely, considering
\eqref{eq:4-spaces}, we assume that 
stars or tops in $\fixspace$ are at least 4-dimensional projective or semiaffine 
but not punctured projective spaces.
In view of \eqref{eq:spineparam} and Table~\ref{tab:subs}, our assumptions 
read as follows
\begin{equation}\label{eq:4-spaces-num}
  4\le n-k\quad\text{and}\quad\dim(W)\neq m+1 
  \qquad\text{or}\qquad 
  4\le k-m \quad\text{and}\quad k\neq m+1.
\end{equation}

\begin{lem}\label{lem:zlepsemibundle2}
  Let $K_i := \LinesOf_{U_i}(X_i)\in \semibundle$, $i=1,2$.
  If $X_1, X_2$ are both stars or tops and\/ $U_1 = U_2$, then\/ 
  $\Upsilon(K_1, K_2)$ and $K_1\cap K_2= \emptyset$.
\end{lem}

\begin{proof}
  By \ref{fact:spineintersections} we get $K_1\cap K_2= \emptyset$.
  Without loss of generality assume that $X_1,X_2$ are stars.
  If $X_1$ is a projective space (i.e.\ it is an\/ $\omega$-star), then we take two distinct 
  projective $\omega$-lines
  $L_1,L_2\in K_1$. 
  Each of $L_1,L_2$ can be extended to the semiaffine $\omega$-top $Y_1,Y_2$, respectively. 
  We have $U_1\in Y_i\cap X_2$ for $i=1,2$, so
  $X_2$ and $Y_i$ share a line $M_i$. Moreover $U_1\in L_i,M_i$, hence 
$L_i\copenc M_i$ and thus $L_i\copla M_i$, so $\Upsilon(K_1, K_2)$ anyway.

  Assume that $X_1$ is a semiaffine space (i.e.\ it is an $\alpha$-star).
  There are two distinct affine lines $L_1,L_2\in K_1$ as $X_1$ is not a punctured projective space.
  As in the first case we extend $L_1,L_2$ to the semiaffine tops and we get our claim.
\end{proof}

We need an equivalence relation to form the bundle of all lines through a given point.
For proper semibundles $K_1, K_2\in\semibundle$ we write
\begin{equation}\label{eq:emptyupsilon}
  \eUps(K_1, K_2)\quad\text{iff}\quad 
    \Upsilon(K_1, K_2), \Upsilon(K_2, K_1),\text{ and either } 
      K_1\cap K_2=\emptyset\text{ or } K_1=K_2.
\end{equation}

To give and idea what $\eUps$ stands for let us summarize \ref{lem:zlepsemibundle1} 
and \ref{lem:zlepsemibundle2}.

\begin{lem}\label{lem:zlepsemibundle}
  Let\/ $K_i := \LinesOf_{U_i}(X_i)\in \semibundle$.
  Then $\eUps(K_1, K_2)$ iff $X_1, X_2$ are both stars or tops and\/ $U_1 = U_2$.
\end{lem}

For a proper semibundle $K\in\semibundle$ we write
\begin{equation}\label{eq:quasi-bundle}
  \klamerka(K, \eUps) := \bigcup\bigl\{ K'\in\semibundle\colon \eUps(K, K')\bigr\}.
\end{equation}
We will show that it is the bundle of all lines through the point determined by 
the semibundle $K$. Thanks to \eqref{eq:4-spaces-num} all stars
or all tops, no matter if they are $\alpha$ or $\omega$, are at least 4-dimensional 
and are not punctured. This is essential here.

\begin{lem}\label{lem:allbundles}
  Let\/ $U$ be a point and $X$ be a maximal strong subspace.
  If\/ $U\in X$, then
  \begin{equation}
    \klamerka(\LinesOf_U(X), \eUps) = \{ L\in\Lines\colon U\in L \}.
  \end{equation}
\end{lem}

\begin{proof}
  The left-to-right inclusion is immediate by \ref{lem:zlepsemibundle}. 
  To show the right-to-left inclusion let $L$ be a line through $U$. 
  By \eqref{eq:4-spaces-num} there is a maximal strong subspace $Y$ of the same type as $X$ 
  which contains $L$ and is not a punctured projective space. 
  Then  $L\in\LinesOf_U(Y)$. 
  Again by \ref{lem:zlepsemibundle} we get $\eUps(\LinesOf_U(X), \LinesOf_U(Y))$
  which makes the proof complete.
\end{proof}

In fact \ref{lem:allbundles} says that $\klamerka(\LinesOf_U(X), \eUps)$ is the bundle
of all lines through $U$.
We can partition the lineset of $\M$ by $\eUps$, so that the equivalence classes will be the points of $\M$.
Note that points $U_1, U_2,\dots U_t$ are
collinear iff $\klamerka(\LinesOf_{U_{1}(X)}, \eUps)\cap\klamerka(\LinesOf_{U_{2}(X)}, \eUps)\cap\dots\cap\klamerka(\LinesOf_{U_t(X)}, \eUps)\neq\emptyset$.
This suffices to state the following theorem.

\begin{thm}\label{thm:main}
  Let\/ $\fixspace$ be a spine space and let\/ $\Lines$ be its lineset. 
  If\/ $\fixspace$ satisfies \eqref{eq:4-spaces-num}, then
  \begin{multline}\label{eq:star}\tag{$\ast$}
    \text{the spine space\/ $\fixspace$ and the structure $\struct{\Lines,\cop}$ of its lines}\\
    \text{equipped with relation $\mathord{\cop}\in\{\copla,\copenc\}$ are definitionally equivalent.}
  \end{multline}  
\end{thm}

Admittedly, our main theorem is proved but it is worth to make some comments regarding geometry 
on lines in $\fixspace$.
There are several geometries on lines that appear throughout this paper: with pencils of lines,
 with bundles of lines, and with binary relation $\cop$.
Comparing \ref{prop:pencpenc} with \ref{thm:main}, note that
significantly weaker assumptions are required to reconstruct pencils of lines than bundles of lines 
in $\struct{\Lines,\cop}$.
It is seen that recovering the geometry of bundles from the geometry of pencils is shorter
and easier than recovering it from the geometry of $\cop$ what we actually did.
However, our goal was to prove some variant of Chow's Theorem (cf.\ \cite{chow}) which says that 
the underlying geometry of a spine space can be defined in terms of binary relations (adjacencies) 
$\copla$ and $\copenc$ on lines.


\section{Excluded cases}

The question now is what about cases excluded by assumptions \eqref{eq:4-spaces-num}. 
The bundle method does not work in these spine spaces, however, for some of them 
we are able to say if \eqref{eq:star} holds or not.
Let us write $w = \dim(W)$.

\ecase{w = n}
In this case $\M$ is a Grassmann space which was treated in \cite{bundle} so, \eqref{eq:star} holds true.

\ecase{w = m = k}
This is a trivial case, entire $\M$ is a single point so, \eqref{eq:star} holds true.

\ecase{w = m = k-1}
In this case $\M$ is a star in $\PencSpace(V, k)$, i.e.\ it is a projective spaces.
Hence, the condition \eqref{eq:star} holds true.

\ecase{w = k+1, m = k}
Dual to the previous case. $\M$ is a top in $\PencSpace(V, k)$, i.e.\ it is a projective spaces.
Again \eqref{eq:star} holds true.

\ecase{w = k, m = k-1}
This time $\M$ is the neighbourhood of a point $W$ in $\PencSpace(V, k)$, i.e.\ the
set of all points that are collinear with $W$.
Maximal strong subspaces of $\M$ are punctured projective spaces arising
from maximal strong subspaces of $\PencSpace(V, k)$ containing $W$. 
If both $X_1, X_2$ are stars or they both are tops, then 
$\cl{X_1}\cap\cl{X_2}=\{W\}$ and $X_1\cap X_2=\emptyset$. 
If $X_1$ is a star and $X_2$ is a top, then
$X_1\cap X_2$ is a line $L$ with $W\in\cl{L}$.

Let $\S$ be the family of all stars and $\T$ be the family of all tops in $\M$.
Consider a star $X$ in $\M$, a homology $\varphi\neq\id$ on $\cl{X}$ with the 
center $W$, and a map $f\colon\Sub_k(V)\lto\Sub_k(V)$ given as follows
$$f(U) = \begin{cases}
  \varphi(U), & U\in X,\\
  U, & U\notin X.
\end{cases}$$
We will write $\Lines_\S$ for the set of all lines contained in stars from $\S$ and
$\Lines_\T$ for the set of all lines contained in tops from $\T$. 
Let $F_\S\colon\Lines_\S\lto\Lines_\S$ be the map induced by $f$ and 
$F_\T\colon\Lines_\T\lto\Lines_\T$ such that $F_\T(L)=L$.

Consider $F:=F_\S\cup F_\T$. It is seen that $F$ is an automorphism of $\struct{\Lines,\cop}$. 
Note that $F(L)\neq L$ iff $L\subset X$ and $W\notin\cl{L}$. 
Take $Y\in\T$, the line $L := X\cap Y$, and a point $U\in L$. 
We have $W\in\cl{L}$, $F(\LinesOf_U(Y)) = \LinesOf_U(Y)$, and
$F(\LinesOf_U(X)) = \LinesOf_{U'}(X)$ for some $U'\neq U$, $U'\in L$. This means that $F$ does not preserve bundles of lines, so 
$F$ is not a collineation of $\M$. Thus, \eqref{eq:star} cannot be
proved in this case.


\appendix

\section{Erratum to: Coplanarity of lines in projective and polar Grassmann spaces \cite{bundle}}\label{sec:erratum}

In the current paper spine spaces are considered in view of coplanarity relation. 
The same problem was discussed in \cite{bundle} for general Grassmann spaces, i.e.\ 
Grassmann spaces and polar Grassmann spaces. Both spines spaces and polar Grassmann spaces
are fragments of Grassmann spaces. 
It has to be noted here that the proof of Lemma~1.3 in \cite{bundle} fails to be
complete.
In this paper, apparently in the context of spine space, we prove that lemma as 
Lemma~\ref{lem:max-copla-cliques}. Its proof is universal in that it refers to 
the ambient Grassmann space. Thus it remains valid for general Grassmann spaces as well.

Roughly speaking the omission in \cite{bundle} was caused by the assumption that 
 every $\copla$-clique is spanned by 3 lines.
In spine spaces, in particular in Grassmann spaces, two distinct $\copla$-cliques 
share at most one line. This is because two strong subspaces share at most a line. 
However, it is different in polar Grassmann spaces, where two
star semibundles can have a lot more in common (cf. \cite{bundle}). 
Assume that all stars as projective spaces have dimension
$d$ and all tops have dimension $d'$. Without loss
of generality we can assume that $d'\leq d$.
In case of polar Grassmann spaces $n=d$ lines are
required in the relation $\TRG_n^{\copla}$ introduced in \eqref{eq:trgcopla}. 
To explain this let us take $t$ lines 
$L_1, L_2,\dots, L_t$, which are pairwise coplanar, but not all lie on a plane. 
Assume that these lines are contained in a star $S$ of dimension $d$ and $t<d$. 
Then, in considered polar geometries there is another star $S'\neq S$ which contains
all the lines $L_1, L_2, \dots, L_t$. We can find two lines $M_1, M_2$ such that 
$M_1\subseteq S$, $M_2\subseteq S'$, $M_1, M_2\nsubseteq S\cap S'$, and 
$M_1, M_2\copla L_1, L_2, \dots, L_t$ but $M_1\not\copla M_2$.
It means that $\TRG_t^{\copla}$ fails to be true for $L_1, L_2,\dots, L_t$.



\begin{flushleft}
  \noindent\small
  K. Petelczyc, M. \.Zynel\\
  Institute of Mathematics, University of Bia{\l}ystok,\\
  K. Cio{\l}kowskiego 1M, 15-245 Bia{\l}ystok, Poland\\
  \verb+kryzpet@math.uwb.edu.pl+,
  \verb+mariusz@math.uwb.edu.pl+
\end{flushleft}

\end{document}